\documentclass[12pt]{amsart}
\usepackage[marginratio=1:1,totalwidth=15.75cm,totalheight=22.275cm]{geometry}
\usepackage[english]{babel}
\usepackage[utf8]{inputenc}
\usepackage{amsmath,amssymb}
\usepackage{orcidlink}
\numberwithin{equation}{section}
\newtheorem{thm}{Theorem}[section]

\newtheorem{prop}[thm]{Proposition}

\newtheorem{con}[thm]{Conjecture}
\newtheorem{rem}[thm]{Remark}
\newtheorem{ques}[thm]{Question}
\newcommand{\C}{\mathbb{C}}
\newcommand{\Ch}{\widehat{\mathbb{C}}}

\newcommand{\eps}{\varepsilon}

\newcommand{\re}{\operatorname{Re}}

\newcommand*{\defeq}{\mathrel{\vcenter{\baselineskip0.5ex\lineskiplimit0pt
  \hbox{\scriptsize.}\hbox{\scriptsize.}}}=}
  
\newcommand{\DD}{\mathbb{D}}
\newcommand{\interior}{\operatorname{int}}
  
\usepackage[shortlabels]{enumitem}

\newcommand{\BU}{\operatorname{\mathit{BU}}}

\begin{document}

\author[D. Mart\'i-Pete]{David Mart\'i-Pete\orcidlink{0000-0002-0541-8364}}

\address{Department of Mathematical Sciences\\ University of Liverpool\\ Liverpool L69 7ZL\\ United Kingdom} 
\email{david.marti-pete@liverpool.ac.uk}

\author[L. Rempe]{Lasse Rempe\orcidlink{0000-0001-8032-8580}} 
\address{\noindent Department of Mathematics \\ The University of Manchester \\ Manchester \\ M13 9PL \\ UK}
\email{lasse.rempe@manchester.ac.uk}

\author[J. Waterman]{James Waterman\orcidlink{0000-0001-7266-0292}}
\address{Center for Computing Sciences\\ Bowie \\ MD 20715 \\ USA}
\email{jawater@super.org}


\date{\today}

\title{A counterexample to Eremenko's conjecture}
\begin{abstract}
The \emph{escaping set} of a transcendental entire function
consists of those points that tend to infinity under iteration.
Eremenko asked whether every connected component of this set 
is unbounded. In~\cite{MRW}, we constructed transcendental entire
functions for which the escaping set has prescribed compact connected components;
in particular, it may have a singleton component. This disproves the
conjecture. In this note, we give an introduction to the problem, aimed
at a general mathematical audience. We also sketch a variant
of the construction in which a prescribed half-strip is an oscillating
wandering domain and an escaping point lies on its boundary. The account
is based on the second author's lecture at the 2026
International Congress of Basic Science. 
\end{abstract}
\maketitle

\section{Introduction}
A \emph{discrete-time dynamical system} is 
described by a set $X$ (the \emph{state space}) and a function
$f\colon X\to X$ (the \emph{transition map}). If the system 
is in state $x\in X$ at a given moment in time, its next state is
given by $f(x)$. Starting at $x\in X$, the consecutive states of
the system follow the \emph{orbit}
\[ 
 x,\quad f(x),\quad f^2(x),\quad\ldots,
\]
where $f^n=f\circ\dots\circ f$ denotes the $n$-th iterate of $f$. 
Given a discrete-time dynamical system, a fundamental question is to
understand the long-term behaviour of orbits, and how points with 
different types of long-term behaviour are organised within the state space $X$.

In \emph{one-dimensional holomorphic dynamics}, the state of the 
system is described by a single complex variable. That is, the
state space $X$ is a \emph{Riemann
surface} (a connected one-dimensional complex manifold) and $f\colon X\to X$ is holomorphic. Interesting behaviour arises in the cases where
$X$ is either the complex plane $\C$, the Riemann sphere $\Ch$, or the
punctured plane $\C^*= \C\setminus\{0\}$, and $f$ is not a conformal isomorphism;
see~\cite[Sections~4--6]{Milnor}.%
\footnote{%
It is also interesting
to consider the dynamics of maps $f\colon U\to X$, where
$U$ is a proper subset of $X$; for example, $X=\Ch$ and $f\colon \C\to\Ch$, in which
case $f$ is called a \emph{meromorphic function}. Here we restrict
to the classical case where $f$ is a self-map of $X$.} 

Despite the apparently simple setting, many fundamental phenomena that occur
  in the general study of dynamical systems already appear in one-dimensional
  holomorphic dynamics, where powerful techniques are available to study them.
  Indeed, complex analysis, complex geometry, planar topology,
  number theory, algebraic geometry and dimension theory all have fruitful interactions with
  holomorphic dynamics.   
  
There are two distinct subfields of holomorphic dynamics: 
  \emph{rational dynamics}, where $X=\Ch$ and $f\colon \Ch\to\Ch$ is 
  a polynomial or a rational map, and \emph{transcendental dynamics}, where
  $X\subsetneq \Ch$ and $f$ has \emph{essential singularities} at the 
  boundary of its domain of definition. Rational dynamics 
  was founded in the early twentieth century by Fatou and Julia;
  we refer to \cite{Milnor} for an introduction. Transcendental dynamics
  was founded by Fatou in 1926~\cite{Fatou1926}. 
   As Fatou observed, the presence of an essential singularity 
  leads to significant differences from rational dynamics. Nonetheless, 
  there is a close interplay between rational and transcendental dynamics,
  with techniques developed in one field often finding applications 
  in the other. For a recent application of ideas from transcendental
  dynamics to \emph{local connectivity of the Mandelbrot set}, a
  central problem of polynomial dynamics, see~\cite{dudkolyubichpacman}.
  
 Like Fatou's paper~\cite{Fatou1926}, the work we describe in this note
  concerns the dynamics of transcendental entire functions 
    \[ f\colon\C\to\C, \]
   i.e., non-polynomial holomorphic self-maps of the complex plane. 
   Well-known examples of transcendental entire functions are given by
   exponential and trigonometric functions. 

Given such a function $f$, the state space -- i.e., the complex plane~-- splits naturally
into two dynamically defined sets. The \emph{Fatou
set} $F(f)$ is the set of starting values at which orbits are \emph{stable}
under small perturbations, while the \emph{Julia set} $J(f)$ 
is its complement. More formally, $F(f)$ is the largest open set on which the family 
$(f^n)_{n=0}^{\infty}$ of iterates is equicontinuous with respect to
spherical distance. The Julia set is always non-empty and uncountable; 
it contains points with dense orbits, and periodic points (points $z$ with
$f^n(z)=z$ for some $n\geq 1$) are also dense in
$J(f)$. (See, for example,~\cite{Bergweiler1993}.) In particular, the dynamics of $f$ 
on $J(f)$ is topologically \emph{chaotic} in the sense of Devaney~\cite{devaneychaos}.

\subsection*{The escaping set and Eremenko's conjecture}
Let $f\colon \C\to\C$ be holomorphic. The \emph{escaping set} of $f$ is an important object, defined by 
\begin{equation}\label{eq:escaping}
 I(f)\defeq\{z\in\C\colon f^n(z)\to\infty\text{ as }n\to\infty\}.
\end{equation}
If $f$ is a polynomial of degree at least $2$, 
then $I(f)$ is its 
\emph{basin of infinity}, an open set on which 
 the iterates converge to infinity locally uniformly. As observed by
 Douady and Hubbard, this basin of infinity is foliated by dynamically natural
 curves called \emph{external rays}. These rays, and partitions of the state
 space generated by rays together with their landing points, are 
 a crucial tool in polynomial dynamics. (See~\cite{milnorpuzzles} for an introduction.) 

From now on, we assume that $f$ is not a polynomial; that is, $f$ is transcendental. For such $f$, the escaping set is more complicated: it is never 
open, and indeed can never be an $F_{\sigma}$ set (a countable union of closed sets);
see~\cite{rempeFsigma}. Nonetheless, in many important cases it contains
structures, such as curves to infinity, that can be used to understand the global
dynamics. (For more background on the escaping set, we refer to the survey~\cite{BergweilerRempe2025}.)

Eremenko~\cite{Eremenko1989} initiated the systematic study of $I(f)$ in 1989.
He proved the following fundamental result.
\begin{thm}[Eremenko, 1989]\label{thm:eremenko}
If $f\colon\C\to\C$ is a transcendental entire function, then $I(f)\neq\emptyset$,
$J(f)=\partial I(f)$, and every connected component of the closure
$\overline{I(f)}$ is unbounded.
\end{thm}
He also stated that it is plausible that $I(f)$ itself has no bounded
connected components. This became known as \emph{Eremenko's conjecture}.
\begin{con}[Eremenko's conjecture]\label{con:eremenko}
Every connected component of the escaping set of a transcendental
entire function is unbounded.
\end{con}

 The escaping set and Eremenko's conjecture have strongly influenced the development of transcendental dynamics since at least the early 2000s. 
  For example, 
  Rippon and Stallard~\cite{RipponStallard2005}
   showed that the so-called \emph{fast escaping set} $A(f)$, introduced by Bergweiler and Hinkkanen~\cite{bergweiler-hinkkanen99}
     has only unbounded connected components. 
     The set $A(f)$ has now become a central object in transcendental dynamics, and also
     in the dynamics of \emph{quasiregular} functions (see~\cite[Sections~5.1~and~10.1]{BergweilerRempe2025}). 
     
 Let us mention four variants of the conjecture.
  \begin{enumerate}[(i)]
   \item \emph{Does the escaping set have at least one unbounded connected component?} This is true; indeed, it follows from the
     above-mentioned result of Rippon and Stallard on $A(f)$.\label{item:unboundedexistence}
   \item \emph{Is $I(f) \cup \{\infty\}$ connected?} This is also true, as established by Rippon and Stallard in~\cite{RipponStallard2011}.
      This statement seems very close to Eremenko's conjecture. However, we note that, for a set $X\subset\C$ that is 
      neither open nor closed, the condition that $X\cup\{\infty\}$ is connected is considerably weaker than that $X$ has only unbounded connected components.\label{item:connectedwithinfinity}
   \item \emph{Is $I(f)\cup \{\infty\}$ \emph{path-connected}?} This stronger version of the 
      question was also posed by Eremenko, and was called the \emph{strong version} of Eremenko's conjecture.
     In~\cite{RRRS}, it was shown that this statement is false, even within the \emph{Eremenko--Lyubich class} of transcendental entire functions
     with a bounded set of singular values. Bishop~\cite{bishop15} later showed that this example can even be realised in the \emph{Speiser class} of functions
     with only finitely many singular values. In fact, this question was one of the original motivations behind Bishop's ground-breaking technique of
     \emph{quasiconformal folding}.\label{item:rrrs}
   \item \emph{If $z\in I(f)$, is there an unbounded connected set $A\subset I(f)$ with $z \in A$ and such that $f^n\to \infty$ uniformly on $A$?} 
      This strengthening of Eremenko's conjecture was suggested to the second author by Eremenko in 2007 (personal 
      communication). This statement
      is also known to be false in the Eremenko--Lyubich and Speiser classes; see~\cite{arclike}.\label{item:uniformeremenko}
  \end{enumerate}
  
Note that the weaker versions~\ref{item:unboundedexistence} and~\ref{item:connectedwithinfinity} of Conjecture~\ref{con:eremenko} hold in general,
while its stronger variants~\ref{item:rrrs} and~\ref{item:uniformeremenko} fail even in restricted function classes. Conjecture~\ref{con:eremenko} itself 
has been verified in many 
settings (again, we refer to~\cite{BergweilerRempe2025} for details). 
Nonetheless, in~\cite[Theorem~1.2]{MRW},  we resolve the conjecture in the negative. In particular, we show the following. 
\begin{thm}[Mart\'i-Pete--Rempe--Waterman, 2025]\label{thm:main}
There exists a transcendental entire function $f$ such that $I(f)$ has a connected component consisting of a single point.
\end{thm}

The example is constructed using approximation theory. A key difficulty lies in controlling the global dynamics
 of the function $f$ sufficiently to ensure that some escaping point $z_0$ cannot be connected to infinity by a connected subset of $I(f)$. 
 Indeed, by the results of~\cite{RipponStallard2011}, there will necessarily be unbounded connected components of the escaping set arbitrarily close to 
 $z_0$, and $z_0$ cannot be separated from these using a bounded closed subset of the plane. (Recall that $I(f)\cup\{\infty\}$ is connected.) So one needs to ensure
 that $z_0$ can be separated from every other escaping point by a suitable  unbounded closed set of non-escaping points. 
 
Techniques that allow this kind of control arose from studying an a priori unrelated question in the topic of \emph{wandering domains} (see Section~\ref{sec:wandering}). 
The goal of this note is to explain how the two are related, and to outline a proof of Theorem~\ref{thm:main}. The proof we present differs from 
that given in~\cite{MRW} in two ways. Firstly, our function is constructed to have a wandering domain, with the escaping point in question
lying on its boundary. This follows an idea sketched in~\cite[Remark~7.4]{MRW}, in
contrast to the construction in~\cite[Theorem~7.1]{MRW}, and simplifies the proof somewhat (subject to citing known results about the
boundary dynamics of wandering domains). Secondly, in~\cite{MRW} the proof is presented in the context of a more general method for realising
wandering compact sets; here we sketch a direct proof of Theorem~\ref{thm:main} only, which again simplifies the presentation.

\subsection*{Basic notation}
 Throughout the paper, $\C$ denotes the complex plane, $\Ch$ the Riemann sphere and $\DD$ the unit disc. Boundaries and closures are considered to be
  taken in $\C$, unless explicitly mentioned otherwise. 

\section{Wandering domains}\label{sec:wandering}
Let $f$ be a transcendental entire function. The connected components of $F(f)$ are called \emph{Fatou components}.
Such a  Fatou component $U$ is \emph{wandering} if its
forward images are pairwise disjoint:
\[
 f^n(U)\cap f^m(U)=\emptyset\qquad(n\neq m).
\]
Sullivan~\cite{Sullivan1985} proved that 
rational maps do not have wandering domains. On the other hand, Baker had already shown in 1976 that
transcendental entire functions may have wandering domains
\cite{Baker1976}.

A wandering domain $U$ is \emph{escaping} if $U\subset I(f)$. It is
\emph{oscillating} if its points have unbounded orbits but do not tend
to infinity. More precisely, such points belong to the bungee set
\[
 \BU(f)\defeq\{z\in\C\setminus I(f)\colon
              \limsup_{n\to\infty}\lvert f^n(z)\rvert =\infty\}.
\]
The first examples of oscillating wandering domains were constructed by
Eremenko and Lyubich~\cite{EremenkoLyubich1987}. Their work introduced
approximation theory as a tool for constructing entire functions with
prescribed dynamical behaviour.

It is well known that all points in a wandering domain $U$ share the same behaviour: if $U\cap I(f)\neq \emptyset$, then $U\subset I(f)$. Likewise, if
$U\cap \BU(f)\neq \emptyset$, then $U\subset \BU(f)$. However, boundary points need not follow the same dynamics. Following
\cite[Definition~1.10]{MRW}, a boundary point whose accumulation behaviour differs from
that of interior points is called a \emph{maverick point} (with respect to $U$). For an
escaping wandering domain, these are precisely the non-escaping boundary
points. For an oscillating wandering domain, an escaping boundary point
is automatically maverick. Such exceptional points are rare from the
viewpoint of \emph{harmonic measure} (a conformally invariant notion of size for subsets of the boundary of a plane domain).

\begin{thm}[Rippon--Stallard, 2011]\label{thm:rippon-stallard}
 Let $f$ be a transcendental entire function.
 If $U$ is an escaping wandering domain of $f$, then $\partial U\cap I(f)$ has full harmonic measure relative to $U$. 
 If $U$ is an oscillating wandering domain of $f$, then $\partial U\cap I(f)$ has zero harmonic measure relative to $U$. 
\end{thm}

More generally, maverick points on the boundary of a wandering domain 
 have harmonic measure zero~\cite[Theorem~1.11]{MRW} (see also~\cite[Theorem~1.3]{osborne-sixsmith16}), but we do not require this more general fact here.
 
\begin{rem}\label{rem:totally-disconnected} The only property of harmonic measure that 
   we will require is of a topological nature. 
    Suppose that $U\subset\C$ is a Jordan domain (a simply connected domain whose boundary in $\Ch$ is a simple closed curve), and let $E\subset \partial U$ have harmonic measure zero. Then $E$ is totally disconnected; that is, $E$ contains no non-degenerate connected subsets. 
    \end{rem}

 Prior to our work in~\cite{MRW}, it was not known whether a wandering domain could have maverick boundary points. Indeed, a question of Rippon (see~\cite[Problem~2.94]{hayman-lingham19}) asked whether there exists a bounded escaping wandering domain having a non-escaping boundary point. 
 One may also ask a variant of this question for oscillating wandering domains.
 
\begin{ques}\label{ques:oscillatingmaverick}
 Does there exist a transcendental entire function $f$ and an oscillating wandering domain $U$ such that $\partial U\cap I(f)\neq \emptyset$?
\end{ques}

The authors' work in~\cite{MRW} was originally motivated by these questions. We identified a construction, using approximation theory,
 that answered both questions. The vertical diameters (i.e., the maximal difference between imaginary parts) of the iterates of this wandering domain
 remain bounded, while the imaginary parts of interior orbits oscillate, having both finite and infinite limit points. On the other hand, the wandering domain
 stretches out in the horizontal direction, with the real parts of interior orbits remaining bounded, but those of a certain boundary point tending to infinity 
 (in the case of Question~\ref{ques:oscillatingmaverick}) or vice versa (for Rippon's question concerning escaping wandering domains). The wandering domain can be constructed to be either bounded (as in Rippon's question) or unbounded as a subset of the plane. 
 
 While working on this construction, we became aware of recent work~\cite{bocthaler} of Boc Thaler, who had used a method similar to ours to realise specific bounded domains 
  as wandering domains of transcendental entire functions; for example, he showed that there exists a transcendental entire function for which the
  unit disc $\DD$ is a wandering domain. We were able to combine his ideas with our method to extend Boc Thaler's result to realise more general
  examples of wandering domains (and, indeed, wandering continua), and to additionally prescribe certain boundary points to be maverick.
  (See~\cite[Theorem~1.7]{MRW} for a precise statement.) In particular, an application of these ideas to unbounded oscillating wandering domains yields the following.
  
  \begin{thm}\label{thm:maverick}
   There is a transcendental entire function $f$ having an unbounded oscillating wandering domain $U$ such that $U$ is a Jordan domain, 
       $\partial U\cap I(f)\neq\emptyset$, and such that furthermore $\overline{U}$ is a connected component of the set of 
       points whose orbits under $f$ are unbounded. 
  \end{thm} 
  
 This not only answers Question~\ref{ques:oscillatingmaverick}, but also establishes Theorem~\ref{thm:main}. 
 
\begin{prop} 
 Suppose that $f$ is as in Theorem~\ref{thm:maverick}, and that $z_0\in \partial U\cap I(f)$. Then $\{z_0\}$ is a connected component of $I(f)$. 
\end{prop}
\begin{proof}
 Let $X$ be the connected component of $I(f)$ containing $z_0$. Then $X\subset \overline{U}$ by the final property in Theorem~\ref{thm:maverick}. 
  On the other hand, $U\subset \BU(f)\subset \C\setminus I(f)$. Hence 
   \[ X\subset \overline{U}\cap I(f) = \partial U \cap I(f), \]
   and the latter set is totally disconnected by Theorem~\ref{thm:rippon-stallard} and Remark~\ref{rem:totally-disconnected}.
   Since $X$ is connected and contains $z_0$, we conclude that $X=\{z_0\}$, as claimed. 
\end{proof} 

\section{Structure of the example} 

 The remainder of the paper is devoted to a sketch of the proof of Theorem~\ref{thm:maverick}. 
  We begin by choosing sequences $(S_j)_{j=0}^{\infty}$ and $(T_j)_{j=0}^{\infty}$ of open horizontal strips with equal heights and 
   pairwise disjoint closures, such that $T_0$ lies above $S_0$, $S_1$ lies above
  $T_0$, $T_1$ above $S_1$, and so on. Choose a horizontal half-strip $U$ with $K\defeq \overline{U}\subset T_0$, as well as points 
  $\zeta\in\partial U$ and $\omega\in U$. Finally, let $D$ be a disc whose closure lies below $\overline{S_0}$. 


\begin{figure}[t]
\centering
\begin{tikzpicture}[x=0.077\textwidth, y=0.077\textwidth, font=\small]

\definecolor{Sfill}{RGB}{236,237,241}   
\definecolor{Tfill}{RGB}{219,223,232}   
\definecolor{Kfill}{RGB}{250,219,172}   
\definecolor{Kline}{RGB}{173,111,22}
\definecolor{Ifill}{RGB}{200,222,244}   
\definecolor{Iline}{RGB}{40,98,158}

\def\XR{12.2}

\def\wig#1#2#3#4#5#6{%
  \fill[Ifill]
     plot[domain=#1:\XR,samples=150,smooth,variable=\x]
       ({\x},{#2-#3+#4*sin(#5*\x+#6)})
  -- plot[domain=#1:\XR,samples=150,smooth,variable=\x]
       ({\XR+#1-\x},{#2+#3+#4*sin(#5*(\XR+#1-\x)+#6)})
  -- cycle;
  \draw[Iline,line width=.45pt]
     plot[domain=#1:\XR,samples=150,smooth,variable=\x]
       ({\x},{#2-#3+#4*sin(#5*\x+#6)});
  \draw[Iline,line width=.45pt]
     plot[domain=#1:\XR,samples=150,smooth,variable=\x]
       ({\x},{#2+#3+#4*sin(#5*\x+#6)});
  \draw[Iline,line width=.8pt]
       (#1,{#2-#3+#4*sin(#5*#1+#6)}) -- (#1,{#2+#3+#4*sin(#5*#1+#6)});
}

\def\zpt#1#2#3#4#5{\fill (#1,{#2+#3*sin(#4*#1+#5)}) circle (1.15pt);}
\def\wpt#1#2#3#4#5{\draw[fill=white,line width=.45pt]
                       (#1,{#2+#3*sin(#4*#1+#5)}) circle (1.15pt);}

\def\fA{80}\def\aA{0.085}\def\pA{0}      
\def\fB{88}\def\aB{0.085}\def\pB{180}    
\def\fC{76}\def\aC{0.085}\def\pC{0}      
\def\fD{72}\def\aD{0.100}\def\pD{0}      
\def\fE{84}\def\aE{0.100}\def\pE{180}    

\def\cK{8.60}   \def\cI{8.12}   \def\cII{7.88}
\def\cIII{6.76} \def\cIV{5.84}  \def\cV{4.46}
\def\cVI{4.25}  \def\cVII{2.75} \def\cVIII{1.20}

\def\wK{8.85}
\def\wI{8.90}   \def\wII{8.72}
\def\wIII{8.94} \def\wIV{8.68}  \def\wV{8.82}
\def\wVI{8.76}  \def\wVII{8.88} \def\wVIII{8.70}

\def\yK{2.14}
\def\yI{0.97}  \def\yIII{0.61} \def\yVI{0.25}   
\def\yIV{4.05} \def\yVII{3.69}                  
\def\yVIII{7.13}                                
\def\yII{5.22}                                  
\def\yV{8.30}                                   

\def\tK{0.120}\def\tOne{0.105}\def\tTwo{0.078}\def\tThree{0.056}

\begin{scope}
\clip (0,-0.14) rectangle (11.5,8.98);

\fill[Sfill] (-1,0.00) rectangle (13,1.20);
\fill[Tfill] (-1,1.54) rectangle (13,2.74);
\fill[Sfill] (-1,3.08) rectangle (13,4.28);
\fill[Tfill] (-1,4.62) rectangle (13,5.82);
\fill[Sfill] (-1,6.16) rectangle (13,7.36);
\fill[Tfill] (-1,7.70) rectangle (13,8.90);
\foreach \yb/\yt in {0.00/1.20,1.54/2.74,3.08/4.28,4.62/5.82,6.16/7.36,7.70/8.90}{%
  \draw[gray!45,line width=.3pt] (-1,\yb)--(13,\yb);
  \draw[gray!45,line width=.3pt] (-1,\yt)--(13,\yt);}

\wig{\cI}{\yI}{\tOne}{\aA}{\fA}{\pA}          
\wig{\cIII}{\yIII}{\tTwo}{\aA}{\fA}{\pA}      
\wig{\cVI}{\yVI}{\tThree}{\aA}{\fA}{\pA}      
\wig{\cIV}{\yIV}{\tTwo}{\aB}{\fB}{\pB}        
\wig{\cVII}{\yVII}{\tThree}{\aB}{\fB}{\pB}    
\wig{\cVIII}{\yVIII}{\tThree}{\aC}{\fC}{\pC}  
\wig{\cII}{\yII}{\tOne}{\aD}{\fD}{\pD}        
\wig{\cV}{\yV}{\tTwo}{\aE}{\fE}{\pE}          

\fill[Kfill] (\cK,{\yK-\tK}) rectangle (13,{\yK+\tK});
\draw[Kline,line width=.55pt] (\cK,{\yK-\tK})--(13,{\yK-\tK});
\draw[Kline,line width=.55pt] (\cK,{\yK+\tK})--(13,{\yK+\tK});
\draw[Kline,line width=1pt]   (\cK,{\yK-\tK})--(\cK,{\yK+\tK});

\zpt{\cI}{\yI}{\aA}{\fA}{\pA}         \wpt{\wI}{\yI}{\aA}{\fA}{\pA}
\zpt{\cIII}{\yIII}{\aA}{\fA}{\pA}     \wpt{\wIII}{\yIII}{\aA}{\fA}{\pA}
\zpt{\cVI}{\yVI}{\aA}{\fA}{\pA}       \wpt{\wVI}{\yVI}{\aA}{\fA}{\pA}
\zpt{\cIV}{\yIV}{\aB}{\fB}{\pB}       \wpt{\wIV}{\yIV}{\aB}{\fB}{\pB}
\zpt{\cVII}{\yVII}{\aB}{\fB}{\pB}     \wpt{\wVII}{\yVII}{\aB}{\fB}{\pB}
\zpt{\cVIII}{\yVIII}{\aC}{\fC}{\pC}   \wpt{\wVIII}{\yVIII}{\aC}{\fC}{\pC}
\zpt{\cII}{\yII}{\aD}{\fD}{\pD}       \wpt{\wII}{\yII}{\aD}{\fD}{\pD}
\zpt{\cV}{\yV}{\aE}{\fE}{\pE}         \wpt{\wV}{\yV}{\aE}{\fE}{\pE}
\fill (\cK,\yK) circle (1.15pt);                                  
\draw[fill=white,line width=.45pt] (\wK,\yK) circle (1.15pt);     
\end{scope}

\node[left=3pt] at (0,0.60) {$S_0$};
\node[left=3pt] at (0,2.14) {$T_0$};
\node[left=3pt] at (0,3.68) {$S_1$};
\node[left=3pt] at (0,5.22) {$T_1$};
\node[left=3pt] at (0,6.76) {$S_2$};
\node[left=3pt] at (0,8.30) {$T_2$};
\node at (5.75,9.22) {$\vdots$};

\node at (10.55,\yK) {$K$};
\node[left=3.5pt]  at (\cK,\yK) {$\zeta$};
\node[above=3.5pt] at (\wK,\yK) {$\omega$};

\node[font=\footnotesize,left=4pt] at (\cI,   {\yI   +\aA*sin(\fA*\cI   +\pA)}) {$f(K)$};
\node[font=\footnotesize,left=4pt] at (\cIII, {\yIII +\aA*sin(\fA*\cIII +\pA)}) {$f^{3}(K)$};
\node[font=\footnotesize,left=4pt] at (\cVI,  {\yVI  +\aA*sin(\fA*\cVI  +\pA)}) {$f^{6}(K)$};
\node[font=\footnotesize,left=4pt] at (\cIV,  {\yIV  +\aB*sin(\fB*\cIV  +\pB)}) {$f^{4}(K)$};
\node[font=\footnotesize,left=4pt] at (\cVII, {\yVII +\aB*sin(\fB*\cVII +\pB)}) {$f^{7}(K)$};
\node[font=\footnotesize,left=4pt] at (\cVIII,{\yVIII+\aC*sin(\fC*\cVIII+\pC)}) {$f^{8}(K)$};
\node[font=\footnotesize,left=4pt] at (\cII,  {\yII  +\aD*sin(\fD*\cII  +\pD)}) {$f^{2}(K)$};
\node[font=\footnotesize,left=4pt] at (\cV,   {\yV   +\aE*sin(\fE*\cV   +\pE)}) {$f^{5}(K)$};

\end{tikzpicture}
\caption{The closed half-strip $K=\overline{U}\subset T_0$ and its images
under~$f$.  Each image is again a topological half-strip extending to the right, and
the images are pairwise disjoint.  Thus $K$ passes through the strips in
the order $S_0$, $T_1$, $S_0$, $S_1$, $T_2$, $S_0$, $S_1$, $S_2$,
$T_3,\dots$, entering each $T_j$ exactly once, at time $N_j$, and
returning to each $S_j$ infinitely often.  The
solid dots mark $\zeta$ and its images, whose real parts tend to
$-\infty$, and the open dots mark $\omega$ and its images, whose real
parts remain bounded.}
\label{fig:strips}
\end{figure}
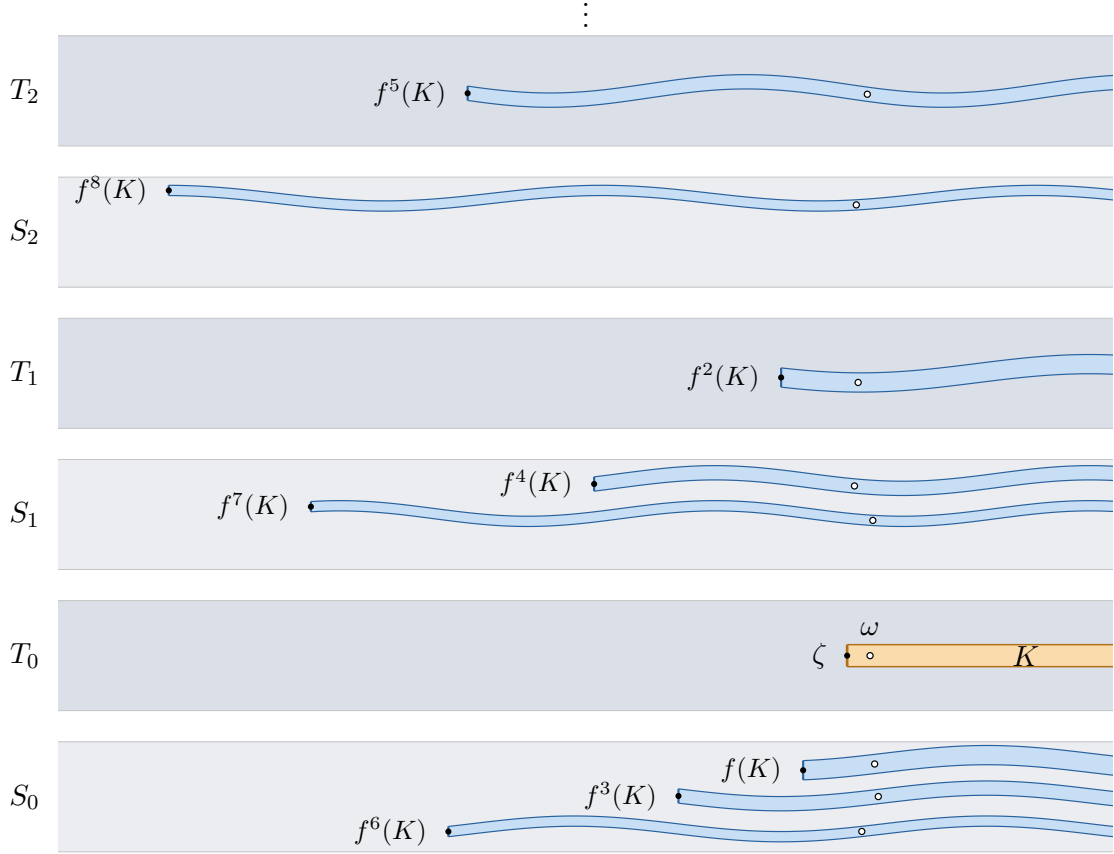

  To prove Theorem~\ref{thm:maverick}, we will construct a transcendental entire function $f$ with the following properties. (See Figure~\ref{fig:strips}.)
  \begin{enumerate}[(1)]
   \item $f(\overline{D})\subset D$.\label{item:Dinvariant}
   \item For every $j\geq 0$, the strip $S_j$ contains an open subset $W_j$ that is mapped conformally onto a horizontal strip containing both $S_{j+1}$ and $T_{j+1}$.
      \label{item:univalentstripmaps}
   \item Under iteration of $f$, the closed half-strip
        $K$ maps through the strips chosen above, in the order $S_0$, $T_1$, $S_0$, $S_1$, $T_2$, $S_0$, $S_1$, $S_2$, $T_3$ etc. That is,
       if we set 
        \[ N_j \defeq \sum_{\ell=1}^j ( \ell + 1 ) = \frac{ j\cdot (j+3)}{2}, \]
        then $f^{N_j}(K) \subset T_j$ and $f^{N_j+k}(K)\subset S_{k-1}$ for $k=1,\dots,j+1$.\label{item:mappingK}
   \item $\re f^n(\zeta)< -j$ for all $N_j < n \leq N_{j+1}$; \label{item:Rezeta}
   \item $\lvert \re f^{n}(\omega)\rvert < 1$ for all $n$.\label{item:Reomega}
  \end{enumerate}
  
 By property~\ref{item:mappingK}, all points in $K$ have unbounded orbits; since periodic points are dense in $J(f)$ and $U\subset K$ is open, it follows that $U\subset F(f)$. 
   Furthermore, by~\ref{item:mappingK}, applied with $k=1$, and \ref{item:Reomega}, the sequence $f^{N_j+1}(\omega)$ is bounded, and hence
   $\omega\in\BU(f)$.     It is known that any Fatou component that contains points of $\BU(f)$ must be a wandering domain, 
  so the Fatou component $\tilde{U}$ containing $U$ is entirely contained in $\BU(f)$ (recall the discussion in
  Section~\ref{sec:wandering}). On the other hand, property~\ref{item:Rezeta} implies that
  $\lvert f^n(\zeta)\rvert\to \infty$ as $n\to\infty$. 
  So $\zeta\in I(f)$ and thus $\zeta\in \partial U\setminus \tilde{U} \subset \partial \tilde{U}$. (In particular, $\zeta$ is maverick for $\tilde{U}$.) 
  To guarantee $\tilde{U}=U$, we also ensure that there is a sequence $(K_j)_{j=0}^{\infty}$ of
  closed half-strips with $K = \bigcap_{j=0}^\infty K_j$ and $K_{j+1}\subset \interior(K_j)$ such that the following hold. 
  
  \begin{enumerate}[(1),resume]
   \item $f^{N_j}(K_j)\subset T_j$ for all $j\geq 0$, and $f^{N_j+k}(K_{j+1})\subset S_{k-1}$ for $k=1,\dots,j+1$.
   \item $f^{N_j}$ is univalent in a uniform Euclidean neighbourhood of $K_j$ (that is, on an open set that contains all points at Euclidean distance at most $\eps$ from $K_j$, for
   some $\eps>0$).
   \item $\overline{f^{N_j+1}(\partial K_j)}\subset D$.\label{item:boundaryinD}
  \end{enumerate}
  
  In other words, the half-strip $K_j$, which is slightly larger than $K$, maps in the same way as $K$ for the first $N_j$ steps, at which point it is contained in the 
  strip $T_j$. Under one further iteration, the boundary of $K_j$ is mapped into the invariant disc $D$, while the slightly smaller set 
  $f^{N_j}( K_{j+1})$ maps univalently into $S_0$, and so on. 
  
  Observe that this implies that $\partial U\subset J(f)$. Indeed, each point of $\partial U = \partial K$ is accumulated on by points of $\partial K_j$, whose orbits
  remain bounded while those on $K$ are unbounded. Hence the family of iterates of $f$ is not equicontinuous at any point of $\partial K$. Moreover, any
  point of $\C\setminus K$ is separated from $K$ by $\partial K_j$ for sufficiently large $j$. Thus (again since these boundaries have 
  bounded orbits), $K=\overline{U}$ is a connected component of the set of points with unbounded orbits.
  
  We conclude that any 
  transcendental entire function $f$ with properties~\ref{item:Dinvariant}--\ref{item:boundaryinD} satisfies the conclusion of Theorem~\ref{thm:maverick}. 

\section{Approximation and an inductive construction}
The principal analytic tool for constructing the desired function $f$ is Arakelyan's approximation theorem~\cite{Arakelyan1964,Gaier}. We
state the special case that is needed here.
\begin{thm}[Arakelyan, 1964]\label{thm:arakelian}
Let $A\subset\C$ be closed, and suppose that the set 
$(\C\setminus A)\cup\{\infty\}$ is connected and locally connected at
infinity. If $g\colon A\to\C$ is continuous on $A$ and holomorphic on
its interior, then, for every $\eps>0$, there exists an entire function
$f\colon\C\to\C$ such that
\[
 \lvert f(z)-g(z)\rvert <\eps
\]
for all $z\in A$. 
\end{thm}
The theorem allows simultaneous
approximation on unbounded sets, which is essential for our construction of unbounded wandering domains. 

A single application of Theorem~\ref{thm:arakelian} is not sufficient.
Instead, we construct $f$ inductively as the locally uniform limit of a sequence of entire functions $(f_k)_{k=0}^{\infty}$.
 Here $f_k$ exhibits the desired mapping behaviour up to and including the strip~$T_k$. More precisely,
 $f_k$ satisfies~\ref{item:Dinvariant}, it satisfies~\ref{item:univalentstripmaps} for~$j\leq k$, and it satisfies 
 the remaining properties for iterates up to and including time $N_k$. 
 
A key observation is the following. If $f_{k+1}$ is chosen sufficiently close to $f_k$ on a lower half-plane $\Delta_k$ containing $T_k$ and all previous strips,
 then $f_{k+1}$ still exhibits the desired mapping behaviour up to and including the strip $T_k$.
 (Strictly speaking, this requires some additional assumptions, concerning the size of $f_k'$ near the sets in question. We refer to~\cite{MRW} for details.) 
 
 To define such an $f_{k+1}$, we choose the half-strip $K_{k+1}\subset K_k$ such that its boundary is very close to that of $K$; this means that the hyperbolic distance in  $\interior(K_{k+1})$ between
 $\zeta$ and $\omega$ can be chosen as large as we want (see below). Then we 
   define $f_{k+1}$ by applying Arakelyan's theorem to a map $g_{k+1}$. This map is defined on the union of $\overline{\Delta_k}$, 
   the strip $\overline{S_{k+1}}$, 
  a closed neighbourhood $R_{k}$ of the curve $f_{k}^{N_k}( \partial K_k)$, and a closed neighbourhood $L_{k+1}$ of $f_{k}^{N_k}( K_{k+1})$. 
  Let us define this union by $A_{k+1}$.
  
  The map $g_{k+1}$ is defined to agree with $f_k$ on $\overline{\Delta_k}$, and to map $R_k$ compactly into the disc $D$. On $\overline{S_{k+1}}$, we choose
   $g_{k+1}$ to be a complex affine map $\alpha_{k+1}$ 
   that maps $S_{k+1}$ to a horizontal strip containing both $\overline{S_{k+2}}$ and $\overline{T_{k+2}}$, and that
   fixes the imaginary axis. 
   
   The definition on $L_{k+1}$ is 
   the most subtle~-- and, crucially, depends inductively on the function $f_k$. 
      By assumption on $f_k$, there is a domain $V_{k+1}\subset S_0$ such that $f_k^{k+1}$ maps $V_{k+1}$ conformally onto $T_{k+1}$. We define
      $g_{k+1}|_{L_{k+1}}$ to be a conformal map that takes values in $V_{k+1}$. If $K_{k+1}$ was chosen
      so that the hyperbolic distance in $\interior(K_{k+1})$ between $\zeta$ and $\omega$ is sufficiently large,
      we can choose this conformal map in such a way that 
       \[ \re g_{k+1}^{N_k+1}(\zeta) =  \re g_{k+1}(f_k^{N_k}(\zeta))<-k \]
        and 
         \[ \lvert \re g_{k+1}^{N_{k+1}}(\omega)\rvert  = \lvert \re f_k^{k+1}( g_{k+1}(f_k^{N_k}(\omega)))\rvert \leq 1 .\]
         
   For $j=0,\dots,k$, consider the set $f_k^j(V_{k+1})\subset S_j$. By induction, on this set $f_k$ is close to the expanding affine map $\alpha_j$ fixing the imaginary axis, and hence it is easy to ensure that 
     $f_k$ maps any point of $g_{k+1}(L_{k+1})$ at real part at least $1$ to a point with larger real part, and any such point at real part
     at most $-1$ to a point with more negative real part. From this, it follows that $g_{k+1}$, and any entire function $f_{k+1}$ sufficiently close to it on
     $A_{k+1}$, has the desired mapping properties up to and including the strip $T_{k+1}$, and the induction can continue. 
     
   We must also anchor the recursion by defining $f_0$ to be a function that is close to an affine map 
   $\alpha_0$ as above on $S_0$ and such that 
   $f_0(\overline{D})\subset D$. Again, such a function is readily obtained by an application of Arakelyan's theorem.
     
  If $f_{k+1}$ approximates $g_{k+1}$ (and hence $f_k$) on $\Delta_k$ with error at most $2^{-k}$, then the maps 
    $(f_k)_{k=k_0}^{\infty}$ form a Cauchy sequence on $\Delta_{k_0}$ for every $k_0$. Hence $(f_k)$ converges locally uniformly on $\C$ 
    to the desired function $f$. The function is transcendental because polynomials do not have bungee points (any unbounded orbit for a polynomial
    is escaping). This concludes our sketch of the construction of $f$, and hence of the proof of Theorems~\ref{thm:main} and~\ref{thm:maverick}.
 
\section*{Authors' statement on the use of generative AI}
Generative AI (Microsoft 365 Copilot) was used to 
 assist with drafting and language
 refinement. Claude Opus 5 also provided minor copy-editing suggestions on a draft of the paper,
 and created Figure~\ref{fig:strips} according to the authors' instructions.
 The authors are fully responsible for the content,
 and AI tools did not supply new mathematical arguments or write substantial parts of the finished article. 

\bibliographystyle{amsalpha}
\bibliography{MRW-FSA-Summary-ICBS-2026}
 \end{document}